\documentclass[11pt]{article}
\usepackage[UKenglish]{isodate}
\usepackage[T1]{fontenc}
\usepackage[noBBpl]{mathpazo}

\usepackage{amsmath, amssymb, amsfonts, amsthm, stmaryrd}

\usepackage[a4paper,left=2.48cm, right=2.48cm, top=2.35cm, bottom=2.35cm]{geometry}
\usepackage{microtype}
\usepackage{titlesec}
\titlespacing*{\section}{0pt}{\baselineskip}{0pt}
\titlespacing*{\subsection}{0pt}{0.5\baselineskip}{0pt}

\usepackage[shortlabels]{enumitem}
\setlist{leftmargin=0.8cm,topsep=0pt,itemsep=-2pt}
\setlist[enumerate]{label=\rm{(\roman*)}}

\makeatletter
\g@addto@macro\normalsize{%
  \setlength\abovedisplayskip{0.4\baselineskip plus 0.4\baselineskip}
  \setlength\belowdisplayskip{0.4\baselineskip plus 0.4\baselineskip}
  \setlength\abovedisplayshortskip{-0.3\baselineskip}
  \setlength\belowdisplayshortskip{0.4\baselineskip plus 0.4\baselineskip}
}
\makeatother

\renewenvironment{thebibliography}[1]
{ \begin{oldthebibliography}{#1}
  \setlength{\parskip}{0pt}
  \setlength{\itemsep}{2pt plus 0.3ex}
  \bgroup\footnotesize }
{ \egroup \end{oldthebibliography} }

\newtheoremstyle{shdefinition}{\topsep}{0.4\topsep}{}{}{\bfseries}{.}{0.5em}{} 
\newtheoremstyle{shplain}{\topsep}{0.4\topsep}{\itshape}{}{\bfseries}{.}{0.5em}{} 
\theoremstyle{shplain}
\newtheorem{thm}{Theorem}
\newtheorem{theorem}{Theorem}[section]
\newtheorem{lemma}[theorem]{Lemma}
\newtheorem{corollary}[theorem]{Corollary}
\theoremstyle{shdefinition}
\newtheorem{rem}{Remark}
\newtheorem{example}[theorem]{Example}
\newtheorem*{definition*}{Definition}

\newcommand{\<}{\langle}
\renewcommand{\>}{\rangle}
\renewcommand{\leq}{\leqslant}
\renewcommand{\geq}{\geqslant}
\newcommand{\leqn}{\trianglelefteqslant}
\newcommand{\GL}{\mathrm{GL}}
\newcommand{\Cyc}{\mathrm{Cyc}}

\begin{document}

\begin{center} 
{\LARGE \textbf{Flexibility in generating sets of finite groups}} \\[11pt]
{\Large Scott Harper}                                             \\[22pt]
\end{center}

\begin{center}
\begin{minipage}{0.8\textwidth}
\small Let $G$ be a finite group. It has recently been proved that every nontrivial element of $G$ is contained in a generating set of minimal size if and only if all proper quotients of $G$ require fewer generators than $G$. It is natural to ask which finite groups, in addition, have the property that any two elements of $G$ that do not generate a cyclic group can be extended to a generating set of minimal size. This note answers the question. The only such finite groups are very specific affine groups: elementary abelian groups extended by a cyclic group acting as scalars. \par
\end{minipage}
\end{center}

\section{Introduction} \label{s:intro}

Generating sets for finite groups have attracted the attention of many authors for decades. In recent years major developments have been made drawing on our extensive knowledge of the generating sets of finite simple groups.  One natural question on this topic is: for which finite groups $G$ is every element contained in a generating set of minimal size $d(G)$? Completing a long line of research, this question was recently proved to admit a very simple answer: the finite groups $G$ such that $d(G/N) < d(G)$ for all $1 \neq N \leqn G$. Acciarri and Lucchini \cite{ref:AcciarriLucchini20} proved the case with $d(G) \geq 3$  and Burness, Guralnick and Harper \cite{ref:BurnessGuralnickHarper21} completed the theorem by proving an influential conjecture of Breuer, Guralnick and Kantor \cite{ref:BreuerGuralnickKantor08} regarding the case with $d(G) \leq 2$.

The most straightforward example with this generation property is the elementary abelian group $p^r$. Considered as a vector space, we see that not only is every nonzero vector contained in a generating set of minimal size (a basis), but moreover any collection of linearly independent vectors can be extended to such a generating set. This motivates the following definition. 

\begin{definition*}
Let $G$ be a finite group. Let $d(G)$ be the minimal size of a generating set for $G$. For an integer $1 \leq k \leq d(G)$, we say that $G$ is \emph{$k$-flexible} if for any $x_1, \dots, x_k \in G$ such that $d(\<x_1,\dots,x_k\>)=k$ there exists $x_{k+1}, \dots, x_{d(G)} \in G$ such that $\<x_1, \dots, x_{d(G)}\>=G$. 
\end{definition*}

We can now rephrase the result of Acciarri--Lucchini \cite{ref:AcciarriLucchini20} and Burness--Guralnick--Harper. \cite{ref:BurnessGuralnickHarper21}.

\begin{thm} \label{thm:1-flexible}
A finite group $G$ is $1$-flexible if and only if $d(G/N) < d(G)$ for all $1 \neq N \leqn G$.
\end{thm}

\begin{rem} \label{rem:1-flexible}
Let us describe the groups in Theorem~\ref{thm:1-flexible}. (We use \textsc{Atlas} notation, so $p$ denotes both a prime and the cyclic group of that order.)
\begin{itemize}
\item[{\rm (i)}]   If $d(G)=1$, then $G = p$ with $p$ prime.
\item[{\rm (ii)}]  If $d(G)=2$, then, by Lemma~2.1 in \cite{ref:BurnessGuralnickHarper21} and the discussion that follows, $G$ is one of:
\begin{itemize}
\item[{\rm (a)}] $p^2$ for $p$ prime
\item[{\rm (b)}] $p^r{:}\<g\>$ for $p$ prime and $g \in \GL_r(p)$ irreducible
\item[{\rm (c)}] $T^r.\<(g,1, \dots, 1)\sigma\>$ for $T$ nonabelian simple, $g \in \mathrm{Aut}(T)$ and an $r$-cycle $\sigma \in S_r$ .
\end{itemize}
\item[{\rm (iii)}] If $d(G) \geq 3$, then, by \cite[Theorem~1.4]{ref:DallaVoltaLucchini98}, $G = \{ (g_1,\dots,g_r) \in L^r \mid Mg_1 = \dots = Mg_r \}$ for $L$ a primitive monolithic group with monolith $M$ and $r = f(d(G)-1)$ where $f$ is the function given in \cite[Theorem~2.7]{ref:DallaVoltaLucchini98}.
\end{itemize}
\end{rem}

This paper shows that startlingly few $1$-flexible groups are also $2$-flexible.

\begin{thm} \label{thm:2-flexible}
Let $G$ be a finite group with $d(G) \geq 3$. Then the following are equivalent:
\begin{itemize}
\item[{\rm (i)}]   $G$ is $1$-flexible and $2$-flexible
\item[{\rm (ii)}]  $G$ is $k$-flexible for all $1 \leq k < d(G)$
\item[{\rm (iii)}] $G = p^r{:}\<g\>$ for a prime $p$, a scalar $g \in \GL_r(p)$ and $r = d(G) - d(\<g\>)$.
\end{itemize}
\end{thm}

While the class of $1$-flexible groups is a large and rich class of soluble and insoluble groups, including all finite simple groups, adding the assumption of $2$-flexibility restricts to a narrow class of supersoluble groups and guarantees $k$-flexibility for all $1 \leq k < d(G)$. 

\begin{rem} \label{rem:2-flexible}
In Theorem~\ref{thm:2-flexible} we assume that $d(G) \geq 3$. For $d(G) = 1$, $2$-flexibility is not interesting. For $d(G) = 2$, we will see in Lemma~\ref{lem:2-flexible_2} that $G$ is $1$-flexible and $2$-flexible if and only if $G = p{:}\<g\>$ where $p$ is prime and $g \in \GL_r(p)$ is a scalar of \emph{prime order}.
\end{rem}

\begin{rem} \label{rem:12-flexible}
Note that $2$-flexibility does not imply $1$-flexibility. For example, if $G$ is the quaternion group $Q_8$, then every proper subgroup of $G$ is cyclic, so $G$ is $2$-flexible, but the nontrivial central element of $G$ is not contained in any generating pair for $G$, so $G$ is not $1$-flexible. However, observe that $Q_8$ is a cyclic central extension of $2^2$, which is $1$-flexible. Indeed, by Corollary~\ref{cor:1-flexible}, every $2$-flexible group is a cyclic central extension of a $1$-flexible group.
\end{rem}

\textbf{\emph{Acknowledgements.}}\ The author thanks the anonymous referee for their useful feedback.

\section{Proofs} \label{s:proofs}

For this section, let $G$ be a finite group such that $d(G) \geq 2$. We begin with some preliminaries.

\begin{lemma} \label{lem:quotient}
Let $1 \leq k \leq d(G)$. Assume that $G$ is $k$-flexible. Let $N \leqn G$ such that $d(G/N) = d(G)$. Then $G/N$ is $k$-flexible.
\end{lemma}

\begin{proof}
Write $d = d(G) = d(G/N)$. Let $Nx_1, \dots, Nx_k \in G/N$ and assume $d(\<Nx_1, \dots, Nx_k\>) = k$. Then $d(\<x_1,\dots,x_k\>) = k$, so, as $G$ is $k$-flexible, there exist elements $x_{k+1}, \dots, x_d \in G$ such that $\<x_1, \dots, x_d \> = G$. Therefore, we deduce $\<Nx_1, \dots, Nx_d \> = G$, so $G/N$ is $k$-flexible.
\end{proof}

To obtain a partial converse to Lemma~\ref{lem:quotient}, in Lemma~\ref{lem:cycliciser_quotient}, we recall that the \emph{cycliciser} of $G$ is 
\[
\Cyc(G) = \{ c \in G \mid \text{$\<c,g\>$ is cyclic for all $g \in G$} \}.
\] 
The fact that the cycliciser $\Cyc(G)$ is a (cyclic) subgroup of $G$ is a consequence of \cite[Lemma~32]{ref:AACNS17}, which states that if $x,y,z \in G$ are such that each of $\<x,y\>$, $\<x,z\>$, $\<y,z\>$ are cyclic, then $\<x,y,z\>$ is also cyclic. Lemma~\ref{lem:cycliciser} characterises $\Cyc(G)$.

\begin{lemma} \label{lem:cycliciser}
The cycliciser $\Cyc(G)$ is the smallest normal subgroup $N \leqn G$ such that $\Cyc(G/N)$ is trivial.
\end{lemma}

\begin{proof}
Let $N = \Cyc(G)$. Let $g \in G$ such that $Ng \in \Cyc(G/N)$. Let $h \in G$ be arbitrary. Then $\<Ng, Nh\>$ is cyclic, so there exists $k \in G$ with $\<Ng,Nh\> = \<Nk\>$. Therefore, $\<N,g,h\> = \<N,k\>$. Now \cite[Lemma~32]{ref:AACNS17} implies that $\<N,k\>$ is cyclic, so $\<N,g,h\>$ is cyclic and thus $\<g,h\>$ is cyclic. This means that $g \in N$, which proves that $\Cyc(G/N)$ is trivial.

Let $N \leqn G$ with $\Cyc(G/N) = 1$. Let $x \in \Cyc(G)$. For all $g \in G$, $\<Nx,Ng\>$ is cyclic as $\<x,g\>$ is cyclic. Hence, $Nx \in \Cyc(G/N)$, so $x \in N$ as $\Cyc(G/N) = 1$. Therefore, $\Cyc(G) \leq N$.
\end{proof}

\begin{lemma} \label{lem:d_quotient}
Assume that $d(G) \geq 2$. Then $d(G/\Cyc(G)) = d(G)$.
\end{lemma}

\begin{proof}
Write $\Cyc(G) = N$ and write $G/N = \< Ng_1, \dots Ng_{d(G/N)} \>$. Then $G = \< N, g_1, \dots, g_{d(G/N)} \>$. As $d(G) \geq 2$, the quotient $G/N$ is nontrivial, so $d(G/N) \geq 1$. Since $\< N, g_1 \>$ is cyclic, fix $h \in G$ such that $\< N, g_1 \> = \< h \>$. Then $G = \< h, g_2, \dots, g_{d(G/N)} \>$, which gives us $d(G) \leq d(G/N)$. Clearly, $d(G) \geq d(G/N)$, so $d(G) = d(G/N)$.
\end{proof}

\begin{lemma} \label{lem:cycliciser_quotient}
Let $2 \leq k < d(G)$. Then $G$ is $k$-flexible if and only if $G/\Cyc(G)$ is $k$-flexible.
\end{lemma}

\begin{proof}
Write $\Cyc(G) = N$. First assume that $G$ is $k$-flexible. Since $d(G) \geq 2$, Lemma~\ref{lem:d_quotient} gives $d(G/N) = d(G)$, so Lemma~\ref{lem:quotient} gives that $G/N$ is $k$-flexible. 

Now assume that $G/N$ is $k$-flexible. Let $x_1, \dots, x_k \in G$ such that $d(\<x_1,\dots,x_k\>)=k$ and write $H = \<x_1,\dots,x_k\>$. The image of $H$ in $G/N$ is isomorphic to $H/(H \cap N)$. Since $H \cap N \leq \Cyc(H)$ and $d(H) = k \geq 2$, Lemma~\ref{lem:d_quotient} implies that $d(\<Nx_1, \dots, Nx_k\>) = d(H/(H \cap N)) = d(H) = k$. Since $G/N$ is $k$-flexible, there exist elements $x_{k+1}, \dots, x_{d(G)} \in G$ such that $\< Nx_1, \dots, Nx_{d(G)} \> = G/N$. Since $\< N, x_{d(G)} \>$ is cyclic, fix an element $y \in G$ such that $\< N, x_{d(G)} \> = \< y \>$. Now $G = \< x_1 \dots, x_{d(G)-1}, y \>$, so $G$ is $k$-flexible.
\end{proof}

The following lemma indicates that the case $k=1$ does not follow the pattern described by Lemma~\ref{lem:cycliciser_quotient}.

\begin{lemma} \label{lem:1-flexible}
Assume that $G$ is $2$-flexible. Then $G$ is $1$-flexible if and only if $\Cyc(G)=1$.
\end{lemma}

\begin{proof}
First assume that $\Cyc(G)$ is trivial. Let $1 \neq x_1 \in G$. As $\Cyc(G)$ is trivial, $x_1 \not\in \Cyc(G)$, so there exists $x_2 \in G$ such that $\<x_1,x_2\>$ is noncyclic. Since $G$ is $2$-flexible, there exist elements $x_3, \dots, x_{d(G)} \in G$ such that $\< x_1, \dots, x_{d(G)} \> = G$. Therefore, $G$ is $1$-flexible.

Now assume that $\Cyc(G)$ is nontrivial. Suppose that $G$ is $1$-flexible. Let $1 \neq x_1 \in \Cyc(G)$. Then there exist elements $x_2, \dots, x_{d(G)} \in G$ such that $\< x_1, \dots, x_{d(G)} \> = G$. Since $x_1 \in \Cyc(G)$, there exists $y \in G$ such that $\<x_1,x_2\> = \<y\>$, so $G = \<x_1, \dots, x_{d(G)} \> = \< y, x_3, \dots, x_{d(G)} \>$, which is impossible as $G$ has no generating set of size $d(G)-1$. Therefore, $G$ is not $1$-flexible.
\end{proof}

\begin{corollary} \label{cor:1-flexible}
Assume that $G$ is $2$-flexible. Then $G/\Cyc(G)$ is $1$-flexible.
\end{corollary}

\begin{proof}
By Lemma~\ref{lem:cycliciser}, $\Cyc(G/\Cyc(G))=1$, so by Lemma~\ref{lem:1-flexible}, $G/\Cyc(G)$ is $1$-flexible.
\end{proof}

We first handle the case where $d(G)=2$.

\begin{lemma} \label{lem:2-flexible_2}
Assume that $d(G) = 2$. Then $G$ is $2$-flexible if and only if $G = p^2$, $G = Q_8$ or $G$ is presented $\< a, b \mid a^p = b^{q^m} = b^{-1}aba^{-r} \>$ for $p \neq q$ primes and $r>1$ satisfying $r \mid q-1$ and $p \mid r^q-1$.
\end{lemma}

\begin{proof}
Observe that $G$ is $2$-flexible if and only if every proper subgroup of $G$ is cyclic. The result now follows from the main theorem of \cite{ref:MillerMoreno1903}.
\end{proof}

\begin{lemma}
Assume that $d(G) = 2$. Then $G$ is $1$-flexible and $2$-flexible if and only if $G = p^2$ or $G = p{:}\<g\>$ where $g \in \GL_1(p)$ has prime order.
\end{lemma}

\begin{proof}
By Remark~\ref{rem:1-flexible} and Lemma~\ref{lem:2-flexible_2}, the groups in the statement are $1$-flexible and $2$-flexible. For the converse, consulting Lemma~\ref{lem:2-flexible_2} and noting that $p^2$ is $1$-flexible and $Q_8$ is not $1$-flexible, it remains to show that if $G = \< a, b \mid a^p = b^{q^m} = b^{-1}aba^{-r} \>$ with $p, q$ distinct primes and $r > 1$ satisfying $r \mid q-1$ and $p \mid r^q-1$, then $m=1$. To do this, note that if $m > 1$, then $1 \neq b^q \in Z(G)$ but $G$ is nonabelian, so $b^q$ is not contained in a generating pair of $G$.
\end{proof}

We now turn to the case where $d(G) \geq 3$. We begin with two examples.

\begin{example} \label{ex:vectorspace}
Let $p$ be prime. Let $G$ be the elementary abelian group $p^r$, so $d(G) = r$. Let $1 \leq k \leq r$. As $p^r$ is a vector space, given $x_1,\dots,x_k \in G$, if $d(\<x_1,\dots,x_k\>) = k$, then $x_1,\dots,x_k$ are linearly independent, so can be extended to a basis $x_1,\dots,x_r$ for $G$. Thus, $G$ is $k$-flexible.
\end{example}

\begin{example} \label{ex:flexible}
Let $r \geq 2$ and $p$ be prime. Let $G = p^r{:}\<g\>$ for a nontrivial scalar $g \in \GL_r(p)$. Note that $d(G) = r+1$ (see \cite[Theorem~2.7]{ref:DallaVoltaLucchini98}). Let $1 \leq k \leq r$. We claim that $G$ is $k$-flexible.

Let $x_1,\dots,x_k \in G$ such that $d(\<x_1,\dots,x_k\>) = k$. Now observe that there are exactly $|p^r|$ distinct $G$-conjugates of $\<g\>$ and that any two of these conjugates intersect trivially. Hence, fix a conjugate $H$ of $\<g\>$ such that $x_1,\dots,x_k\ \not\in H$. Write $x_i = n_ih_i$ where $n_i \in p^r$ and $h_i \in H$; note that $n_i \neq 1$ as $x_i \not\in H$. Fix $x_{k+1}, \dots, x_r \in p^r$ such that $\<n_1,\dots,n_k,x_{k+1},\dots,x_r\> = p^r$ and let $x_{r+1}$ generate $H$. Let $X = \<x_1,\dots,x_{r+1}\>$. Then $x_{r+1} \in X$, so $H \leq X$, and for each $1 \leq i \leq k$, we have $h_i \in \<x_{r+1}\> \leq X$ and hence $n_i = x_ih_i^{-1} \in X$, so $p^r = \< n_1, \dots, n_k,x_{k+1}, \dots, x_r\> \leq X$, from which we can conclude that $G = p^r{:}H = X = \< x_1, \dots, x_{r+1} \>$. 
\end{example}

\begin{corollary}\label{cor:flexible}
Assume $d(G) \geq 3$ and $G/\Cyc(G)$ is $p^r{:}\<g\>$ where $p$ is prime and $g \in \GL_r(p)$ is a scalar. Then $G$ is $k$-flexible for all $2 \leq k < d(G)$.
\end{corollary}

\begin{proof}
This is a consequence of Lemma~\ref{lem:cycliciser_quotient} and Examples~\ref{ex:vectorspace} and~\ref{ex:flexible}.
\end{proof}

We now show Examples~\ref{ex:vectorspace} and~\ref{ex:flexible} to be, in essence, the only examples.

\begin{lemma} \label{lem:2-flexible}
Assume $d(G) \geq 3$ and $G$ is $2$-flexible. Then every minimal normal subgroup of $G$ is cyclic.
\end{lemma}

\begin{proof}
Let $N$ be a minimal normal subgroup of $G$. Towards a contradiction, suppose that $N$ is noncyclic. By \cite[Lemma~32]{ref:AACNS17}, there exist elements $x_1,x_2 \in N$ such that $\<x_1,x_2\>$ is noncyclic. By assumption, there exists $x_3, \dots, x_{d(G)} \in G$ such that $\<x_1,\dots,x_{d(G)}\> = G$, so $\<Nx_1, \dots, Nx_{d(G)}\> = \< Nx_3, \dots, Nx_{d(G)}\> = G/N$, which implies that $d(G/N) \leq d(G) - 2$. Now the main theorem of \cite{ref:Lucchini95} implies that $G$ is a nonabelian simple group, which is a contradiction since $d(G) \geq 3$. We conclude that $N$ is cyclic.
\end{proof}

\begin{lemma}\label{lem:12-flexible}
Assume $d(G) \geq 3$ and $G$ is both $1$-flexible and $2$-flexible. Then $G = p^r{:}\<g\>$ for a prime $p$, a scalar $g \in \GL_r(p)$ and $r = d(G) - d(\<g\>)$.
\end{lemma}

\begin{proof}
As $G$ is $1$-flexible, as discussed in Remark~\ref{rem:1-flexible}, we know $d(G/N) < d(G)$ for all $1 \leq N \leqn G$, so by \cite[Theorem~1.4]{ref:DallaVoltaLucchini98}, 
\[
G = \{ (g_1,\dots,g_r) \in L^r \mid Mg_1 = \dots = Mg_r \}
\]
with $L$ a primitive monolithic group with monolith $M$ and $r = f(d(G)-1)$ where $f$ is the function given in \cite[Theorem~2.7]{ref:DallaVoltaLucchini98}. Since $G$ is $2$-flexible, by Lemma~\ref{lem:12-flexible}, every minimal normal subgroup of $G$ is cyclic, so $M = p$  and we deduce that $L = p{:}\<\lambda\>$ where $\lambda \in \GL_1(p)$, so $G = p^r{:}\<g\>$ where $g \in \GL_r(p)$ is a scalar. It remains to note that either $g=1$ and $r=d(G)$ or otherwise $g \neq 1$ and \cite[Theorem~2.7]{ref:DallaVoltaLucchini98} implies that $r = d(G)-1$.  
\end{proof}

\begin{theorem} \label{thm:flexible}
Let $G$ be a finite group with $d(G) \geq 3$. Then the following are equivalent:
\begin{itemize}
\item[{\rm (i)}]   $G$ is $2$-flexible
\item[{\rm (ii)}]  $G$ is $k$-flexible for all $2 \leq k < d(G)$
\item[{\rm (iii)}] $G/\Cyc(G) = p^r{:}\<g\>$ for prime $p$, scalar $g \in \GL_r(p)$ and $r = d(G) {-} d(\<g\>)$.
\end{itemize}
\end{theorem}

\begin{proof}
(iii) $\implies$ (ii): This is Corollary~\ref{cor:flexible}. (ii) $\implies$ (i): This is immediate. (i) $\implies$ (iii): Lemma~\ref{lem:quotient} implies $G/\Cyc(G)$ is $2$-flexible and Lemma~\ref{lem:cycliciser} implies $\Cyc(G/\Cyc(G))$ is trivial, whence Lemma~\ref{lem:1-flexible} implies $G/\Cyc(G)$ is $1$-flexible, so Lemma~\ref{lem:12-flexible} implies $G/\Cyc(G)$ has the claimed structure. 
\end{proof}

\begin{proof}[Proof of Theorem~\ref{thm:2-flexible}]
This is an immediate consequence of Lemma~\ref{lem:1-flexible} and Theorem~\ref{thm:flexible}.
\end{proof}

\vspace{11pt}

\noindent Scott Harper \newline
School of Mathematics, University of Bristol,  BS8 1UG, UK \newline
Heilbronn Institute for Mathematical Research, Bristol, UK \newline
\texttt{scott.harper@bristol.ac.uk}

\end{document}